\documentclass[11pt]{article}
\usepackage{amsmath} % \font used for R in Real numbers
\usepackage{verbatim,empheq}
\usepackage{amssymb, float}
\usepackage{color}
\usepackage{graphicx,epsfig}
\graphicspath{{figures/}}
\usepackage[applemac]{inputenc}
\title{ On Properties of Differential Inclusions with Prox-regular Sets }

\def\beq{\begin{equation}}
\def\eeq{\end{equation}}
\def\baq{\begin{eqnarray}}
\def\eaq{\end{eqnarray}}
\def\baqn{\begin{eqnarray*}}
\def\eaqn{\end{eqnarray*}}

\newcommand{\R}{\mathbb{R}}

\newcommand{\ball}{\mathbb{B}}

  \newcommand{\Multi}{\,\,\lower 1pt
     \hbox{$\overrightarrow{\longrightarrow}$}\,\,}
     \newcommand{\multi}{\,\,\lower 1pt
     \hbox{$\overrightarrow{\rightarrow}$}\,\,}
\def\image #1 (#2,#3) (echelle #4) #5{
\dimen2=#2
\dimen3=#3
\divide \dimen2 by 1000
\multiply \dimen2 by #4
\divide \dimen3 by 1000
\multiply \dimen3 by #4
\setbox1 =\vbox to \dimen2{\hsize=\dimen3\vfill\special{picture #1
scaled #4}}
\vbox{\hsize=\dimen3\box1\medskip\centerline{#5}}
}
\newtheorem{definition}{Definition}[section]
\newtheorem{prop}{Proposition}[section]
\newtheorem{thm}{Theorem}[section]

\newtheorem{lemma}{Lemma}[section]

\begin{document}

\author{Ba Khiet \sc Le\\
{\small \it Centro de Modelamiento Matem\'atico (CMM), Universidad de Chile, Santiago, Chile}\\
{\small  lkhiet@dim.uchile.cl} \\
}

%\author{
%\author{Samir \sc Adly\\
%{\small \it XLIM UMR-CNRS 7252, Universit\'e de Limoges, 87060 Limoges, France}\\
%{\small  samir.adly@unilim.fr} \\
%\\
%\\
%Ba Khiet \sc Le\\
%{\small \it Centro de Modelamiento Matem\'atico (CMM),
%Universidad de Chile, Santiago, Chile}\\
%{\small ahantoute@dim.uchile.cl
%  } \\
%  }}
%\\
%\\
%  Lionel \sc Thibault\\
%{\small \it D\'epartement de Math\'ematiques, Universit\'e Montpellier II, 34095 Montpellier cedex 5, France}\\
%{\small  lionel.thibault@math.univ-montp2.fr}
%  }

%\address[authorlabel1]{IREMIA, Universit\'e de La R\'eunion, 97400 Saint-Denis, France}
%\address[authorlabel2]{DMI-XLIM, Universit\'e de Limoges, 87060 Limoges, France}
%\address[authorlabel3]{INRIA Rh\^ones-Alpes, 38334 Saint-Ismier, France}

\maketitle
\begin{abstract}

In this paper,   some regularity properties of solutions of  the following differential inclusion
\begin{equation}\nonumber
\left\{
\begin{array}{l}
\dot{x}(t) \in f\big(x(t)\big) -N_{C}\big(x(t)\big)\; {\rm a.e.} \; t \in [0,+\infty),\\ \\
x(0) = x_0\in C,
\end{array}\right.
\end{equation}
are analyzed where $f: H\to H$ is Lipschitz continuous  and $C$ is closed, uniformly prox-regular subset of a Hilbet space $H$. Here $N_{C}(\cdot)$ denotes the  proximal normal cone of $C$.  This work can be considered as an improvement of \cite{hm} since these properties are established without the additional tangential condition at each point in $C$. 
\end{abstract}

\noindent {\bf AMS subject classications:} 34A60, 49J52, 49J53.
\section{Introduction}
In the seventies, sweeping processes are introduced and  deeply studied by J. J. Moreau through the series of papers \cite{M1,M2,M3,M4,M5} which plays an important role in elasto-plasticity, quasi-statics, dynamics, especially in mechanics \cite{M56,M6,Brogliato}. Roughly speaking, a point is swept by a moving closed convex set $C(t)$ in a Hilbert space $H$ and can be formulated in the form of differential inclusion as follows
\begin{equation}
\left\{
\begin{array}{l}
\dot{x}(t) \in -N_{C(t)}(x(t))\; {\rm a.e.} \; t \in [0,T],\\ \\

x(0) = x_0\in C(0),
\end{array}\right.
\end{equation}
where $N_{C(t)}(\cdot)$ denotes the normal cone of $C(t)$ in the sense of convex analysis. When the systems are perturbed, it is natural to study the following variant
\begin{equation}
\left\{
\begin{array}{l}
\dot{x}(t) \in -N_{C(t)}(x(t))+F(t,x(t))\; {\rm a.e.} \; t \in [0,T],\\ \\

x(0) = x_0\in C(0),
\end{array}\right.
\end{equation}
where $F: \R^+\times H \to 2^H$  is a set-valued mapping with nonempty weakly compact convex values in $H$.
 For example, to study the planning procedures in mathematical economy, C. Henry \cite{Henry} introduced and proved the existence of solutions  in finite dimension of the system
\begin{equation}\label{henry}
\left\{
\begin{array}{l}
\dot{x}(t) \in  P_{T_{C}\big(x(t)\big)}\big(F(x(t))\big)\; {\rm a.e.} \; t \in [0,T],\\ \\
x(0) = x_0\in C,
\end{array}\right.
\end{equation}
where $F:\R^n\to 2^{\R^n}$ is upper semi-continuous with nonempty, convex, compact values and $C$ is a closed, convex set in $\R^n$. Here $T, P$ denote the tangent cone and projection operators, respectively. Later B. Cornet \cite{Cornet} extended the system $(\ref{henry})$ for the case $C\subset\R^n$ is Clarke tangentially regular and reduced to 
\begin{equation}
\left\{
\begin{array}{l}
\dot{x}(t) \in F\big(x(t)\big) -N_{C}\big(x(t)\big)\; {\rm a.e.} \; t \in [0,T],\\ \\
x(0) = x_0\in C.
\end{array}\right.
\end{equation}
 There are numerous results for various variants of sweeping processes in literature but most of them are about the existence of solutions (see, e.g., \cite{Bounkhel,ET,ET1,t2}). 
In this paper, we are interested in properties of solutions of the differential inclusion 
\begin{equation}\label{main}
\left\{
\begin{array}{l}
\dot{x}(t) \in f\big(x(t)\big) -N_{C}\big(x(t)\big)\; {\rm a.e.} \; t \in [0,+\infty),\\ \\
x(0) = x_0\in C,
\end{array}\right.
\end{equation}
 where $f: H\to H$ is Lipschitz continuous  and $C$ is closed, uniformly prox-regular subset of a Hilbet space $H$. It is known that  $(\ref{main})$ has a unique locally absolutely continuous solution $x(\cdot)$ on $ [0,+\infty)$ (see \cite{ET} for example).
However, it is also important to know more regularity properties of solutions, even the asymptotic behaviour, to understand better the systems. In \cite{hm}, the authors considered this direction for the same problem. The main properties are the right differentiable of the solution and $\dot{x}^+(\cdot)$ is right continuous at each $t\ge 0$, which later play an important role in studying Lyapunov functions as well as asymptotic behaviour of solutions. However, these properties are obtained in \cite{hm} under  the tangential condition: $f(x)\in T(C,x)$ for all $x\in C.$ The condition is unnecessary since if $C$ is closed, convex then $N_C(\cdot)$ is maximal monotone operator and thus we do not need such kind of condition \cite{Brezis}. It  motivates us to establish the same properties but without the additional tangential condition.

The paper is organized as follow.  In section \ref{section2}, we recall some basic notations, definitions and results which are used throughout the paper.  Some regularities properties of solutions are established without tangential  condition in section $\ref{section3}$.
Some conclusions and perspectives end the paper in section \ref{section4}.
\section{Notations and Mathematical Background}
Let us  begin with some notations used in  the paper. Let $H$ be a Hilbert space. Denote by $\langle\cdot,\cdot\rangle$ , $\|\cdot\|$ the scalar product and the corresponding norm in $H$.  Denote by $I$ the identity operator, by  $\ball$ the unit ball in $H$ and $\ball_r=r\ball,\; \ball_r(x)=x+r\ball$. The distance from a point $s$ to a closed set $C$ is denoted by ${ d}(s,C)$ or ${ d}_C(s)$ and
$${ d}(s,C)=\inf_{x\in C} \|s-x\|.$$
Denote by $C^0$ the set of minimal norm elements of $C$, $i. e.$
$$C^0=\{c\in C: \|c\|=\inf_{c'\in C} \|c'\|\}.$$
It is know that if $C$ is closed and convex then $C^0$ contains exactly one element.
The set of all points in $C$ that are nearest to $s$ is denoted by
$${\rm Proj}(C,s)=\{x\in C: \|s-x\|={ d}(s,C)\}.$$
When ${\rm Proj}(C,s)=\{x\}$, we can write $x={\rm proj}(C,s)$ to emphasize the single-valued property. Let $x\in {\rm Proj}(C,s)$ and $t\ge 0$, then the vector $t(s-x)$ is called \textit{proximal normal} to $C$ at $x$. The set of all such vectors is a cone, called \textit{proximal normal cone} of $C$ at $x$ and denoted by $N^P(C,x)$. It is a known result \cite{Clarke,Rockafellar1} that $\xi\in N^P(C,x)$ if and only if there exist some $\sigma>0, \delta>0$ such that
$$\langle \xi, y-x \rangle \le \delta \|y-x\|^2\;\;{\rm for\;all\;}y\in C\cap \ball_\sigma(x).$$
The  Fr\'echet normal cone $N^F(\cdot)$, the limiting normal cone $N^L(\cdot)$ and the Clarke normal cone $N^C(\cdot)$ are defined respectively as follows:
$$N^F(C,x)=\{\xi\in H: \forall \delta>0, \exists \sigma>0 \;s.\; t.\; \langle \xi, y-x \rangle \le \delta \|y-x\|\;{\rm for\;all\;}y\in C\cap \ball_\sigma(x)\}.$$
\baqn
N^L(C,x)&=&\{\xi\in H:\exists\; \xi_n\to \xi\;{\rm weakly\;and}\; \xi_n\in N^P(C,x_n), x_n\to x\;{\rm in}\; C\}\\
&=&\{\xi\in H:\exists\; \xi_n\to \xi\;{\rm weakly\;and}\; \xi_n\in N^F(C,x_n), x_n\to x\;{\rm in}\; C\}.
\eaqn
$$N^C(C,x)=\overline{{\rm co}}N^L(C,x).$$
If $x\notin C$, one has $N^P(C,x)=N^F(C,x)=N^L(C,x)=N^C(C,x)=\emptyset$
and  for all $x\in C$:
$$N^P(C,x)\subset N^F(C,x)\subset N^L(C,x)\subset N^C(C,x).$$
If $C$ is convex then these normal cones are coincide. It is in fact still true for prox-regular sets, which are defined as follows. Then we can write only $N(C,x)$ for simplicity.
\begin{definition} The closed set $C$ is called $r-prox-regular$ iff  each point $s$ in the $r$-enlargement of $C$
$$U_{r}(C)=\{w\in H: { d}(w,C)<r\},$$
has a unique nearest point ${\rm proj}(C,s)$ and the mapping ${\rm proj}(C,\cdot)$ is continuous in $U_{r}(C)$.
\end{definition}
\begin{prop} \cite{Poliquin,t2}\label{propprox} Let $C$ be a closed set in $H$. The followings are equivalent:\\
1) $C$ is $r-prox-regular$.\\
2) For all $x\in C$ and $\xi \in N^L(C,x)$ such that $\|\xi\|\le r$, we have
\beq
x={\rm proj}(C,x+\xi).
\eeq
3) For all $x\in C$ and $\xi\in N^L(C,x)$, we have
$$\langle \xi, y-x\rangle \le \frac{\|\xi\|}{2r}\|y-x\|^2\;\;\forall \; y\in C.$$
4) (Hypo-monotonicity) For all $x,x'\in C$, $\xi\in N^L(C,x)$, $\xi'\in N^L(C,x')$  and $\xi,\xi'\in \ball_{r}$ we have
$$\langle \xi-\xi',x-x' \rangle \ge -\|x-x'\|^2.$$
\end{prop}
If $r=+\infty$, then $C$ is convex. Some examples of prox-regular sets \cite{Bounkhel}: 
\begin{enumerate}
\item The finite union of disjoint intervals is non-convex but uniformly $r$-prox-regular
and $r$ depends on the distances between the intervals.
\item More generally, any finite union of disjoint convex subsets in $H$ is non-convex but
uniformly $r$-prox-regular and $r$ depends on the distances between the sets.
\end{enumerate}

We finish the section with a version of Gronwall's inequality (see, e.g., Lemma 4.1 in \cite{Showalter}).
\begin{lemma}\label{gronwall}
Let $T>0$ be given and $a(\cdot),b(\cdot)\in L^1([t_0,t_0+T];\R)$ with $b(t)\ge 0$ for almost all $t\in [t_0,t_0+T].$ Let the absolutely continuous function $w: [t_0,t_0+T]\to \R_+$ satisfy:
\beq
(1-\alpha)w'(t)\le a(t)w(t)+b(t)w^\alpha(t),\;\; a. e. \;t\in [t_0,t_0+T],
\eeq
where $0\le \alpha<1$. Then for all $t\in [t_0,t_0+T]$:
\beq
w^{1-\alpha}(t)\le w^{1-\alpha}(t_0){\rm exp}\Big(\int_{t_0}^t a(\tau)d\tau\Big)+\int_{t_0}^t{\rm exp}\Big(\int_{s}^t a(\tau)d\tau\Big)b(s)ds.
\eeq
\end{lemma}
\label{section2}
\section{Main Results}
\label{section3}
Let us first recall the existence and uniqueness result of $(\ref{main})$ (see, e.g.,  \cite{ET}).
\begin{thm}
Let $H$ be a Hilbert space and $C$ be a closed, r-prox-regular set. Let $f:H\to H$ be a k-Lipschitz continuous function. Then for each $x_0\in C$, the following differential inclusion 
\begin{equation}
\left\{
\begin{array}{l}
\dot{x}(t) \in f\big(x(t)\big) -N_{C}\big(x(t)\big)\; {\rm a.e.} \; t \in [0,+\infty),\\ \\
x(0) = x_0\in C,
\end{array}\right.
\end{equation}
has a unique locally absolutely continuous solution $x(\cdot)$. In addition, we have
\beq
\|\dot{x}(t) -f\big(x(t)\big)\|\le \|f\big(x(t)\big)\|\; {\rm for}\; a. e. \; t\ge 0.
\eeq
\end{thm}
Let $x(\cdot)$ be the unique solution of  $(\ref{main})$ satisfying $x(0)=x_0$. Define $v: \R_+\to H$ by $v(t):=\Big(f\big(x(t)\big)-N\big(C,x(t)\big)\Big)^0$ and $v_0:=v(0)=\big(f(x_0)-N(C,x_0)\big)^0.$ By using similar arguments as in Lemma 1.8 \cite{MT}, we have the following lemma.
\begin{lemma}\label{lsc}
 We have 
\beq
\|v_0\|\le \liminf_{t\to 0^+}\|v(t)\|.
\eeq
\end{lemma}
\begin{proof}
If $\liminf_{t\to 0^+}\|v(t)\|=+\infty$ then the conclusion holds. If $\liminf_{t\to 0^+}\|v(t)\|=\gamma< +\infty,$ then there exists a sequence $(t_n)_{n\ge 1}$ such that $t_n \to 0^+$ and $\lim_{n\to +\infty}\|v(t_n)\|=\gamma.$ In particular, the sequence $\big(v(t_n)\big)_{n\ge 1}$ is bounded hence there exist a subsequence $\big(v(t_{n_k})\big)_{k\ge 1}$ and $\xi\in H$ such that $\big(v(t_{n_k})\big)_{k\ge 1}$ converges weakly to $\xi$. Recall that 
$$v(t_{n_k})=\Big(f\big(x(t_{n_k})\big)-N\big(C;x(t_{n_k})\big)\Big)^0\in f\big(x(t_{n_k})\big)-N\big(C;x(t_{n_k})\big).$$
Hence $f\big(x(t_{n_k})\big)-v(t_{n_k})\in N\big(C;x(t_{n_k})\big).$ We can find some $\beta>0$ such that $\|f\big(x(t_{n_k})\big)-v(t_{n_k})\|\le \beta$ for all $k\ge 1.$ Using the prox-regularity of $C$, one has
\beq
\langle f(x\big(t_{n_k})\big)-v(t_{n_k}), c-x(t_{n_k}) \rangle \le \frac{\beta}{2r}\|c-x(t_{n_k})\|^2 \;{\rm for\; all}\; c\in C, \;k\ge 1.
\eeq
Let $k\to +\infty$, we get
\beq
\langle f(x_0)-\xi,c-x_0\rangle \le \frac{\beta}{2r}\|c-x_0\|^2 \;{\rm for\; all}\; c\in C.
\eeq
Thus $f(x_0)-\xi\in N(C;x_0)$ or equivalently $\xi\in f(x_0)-N(C;x_0)$. Then
\beq
\|\xi\|\le \liminf_{k\to+\infty}\|v(t_{n_k})\|= \liminf_{n\to+\infty}\|v(t_n)\|=\gamma,
\eeq
due to the weak lower semicontinuity of the norm and the conclusion follows.
\end{proof}
\begin{lemma}\label{limsup}
Let $x(\cdot)$ be the unique solution of  $(\ref{main})$ satisfying $x(0)=x_0$. Then one has
\beq
\limsup_{t\to 0^+}\|\frac{x(t)-x_0}{t}\|\le \|v_0\|,
\eeq
where $v_0=\big(f(x_0)-N(C,x_0)\big)^0=f(x_0)-{\rm proj}\big(f(x_0), N_C(x_0)\big).$
\end{lemma}
\begin{proof}
We have
\begin{equation}
\left\{
\begin{array}{l}
\dot{x}(t) - f\big(x(t)\big)\in- N_{C}\big(x(t)\big)\; {\rm a.e.} \; t \in [0,+\infty),\\ \\
v_0-f(x_0)\in - N_{C}(x_0),
\end{array}\right.
\end{equation}
and $\|\dot{x}(t) -f\big(x(t)\big)\|\le \|f\big(x(t)\big)\|\; {\rm for}\; a. e. \; t\ge 0.$
Using the prox-regularity of $C$ and Proposition \ref{propprox}, one has
\beq
\langle \dot{x}(t)-f\big(x(t)\big)-v_0+f(x_0), x(t)-x_0 \rangle \le \frac{1}{r}\big(\|f\big(x(t)\big)\|+\|v_0-f(x_0)\|\big)\|x(t)-x_0\|^2.
\eeq
Combining with the k-Lipschitz continuity of $f(\cdot)$, one deduces that
\beq
\frac{1}{2}\frac{d}{dt}\|x(t)-x_0\|^2\le \|v_0\|\|x(t)-x_0\|+a(t)\|x(t)-x_0\|^2,
\eeq
where $a(t)=k+\frac{1}{r}\big(\|f\big(x(t)\big)\|+\|v_0-f(x_0)\|\big)$. Using Gronwall's inequality (Lemma \ref{gronwall}), one obtains for all $t\ge 0$ that
\beq
\|x(t)-x_0\|\le \|v_0\| \int_0^t {\rm exp}\Big(\int_s^t a(\tau)d\tau\Big)ds.
\eeq
Hence
\beq
\limsup_{t\to 0^+}\|\frac{x(t)-x_0}{t}\|\le \|v_0\|\limsup_{t\to 0^+}\frac{1}{t}\int_0^t {\rm exp}\Big(\int_s^t a(\tau)d\tau\Big)ds=\|v_0\|.
\eeq
\end{proof}
\begin{lemma}\label{es2sol}
Let $x(\cdot),y(\cdot)$ be the unique solution of $(\ref{main})$ satisfying initial conditions $x(0)=x_0,y(0)=y_0$ respectively. Then for all $t\ge 0:$
\beq\label{2sol}
\|x(t)-y(t)\|\le\|x(0)-y(0)\| {\rm exp}\Big(\int_0^tb(s)ds\Big)\;\;t\ge 0,
\eeq
where $b(t)=k+\frac{1}{r}(\|f(x(t))\|+\|f(y(t))\|).$
In particular, for $a. e.\;t\ge 0$, one has
\beq\label{boundedde}
\|\dot{x}(t)\|\le \|v_0\|{\rm exp}\Big(\int_0^t\big(k+\frac{2\|f(x(s))\|}{r}\big)ds\Big),
\eeq
where $v_0$ is defined in Lemma \ref{limsup}.
\end{lemma}
\begin{proof}
Using the prox-regularity of $C$ and Lipschitz continuity of $f(\cdot)$ similarly as above, we have
\beq
\frac{1}{2}\frac{d}{dt}\|x(t)-y(t)\|^2\le b(t)\|x(t)-y(t)\|^2\;\;a. e. \; t\ge0,
\eeq
where $b(t)=k+\frac{1}{r}\big(\|f\big(x(t)\big)\|+\|f\big(y(t)\big)\|\big).$ Then the Gronwall's inequality (Lemma \ref{gronwall}) implies $(\ref{2sol})$. Given some $h>0$, and we take $y(0)=x(h)$ then $y(t)=x(t+h)$ for all $t\ge 0.$ From $(\ref{2sol})$, we deduce that 
\beq\label{estidiff}
\|\frac{x(t+h)-x(t)}{h}\|\le \|\frac{x(h)-x(0)}{h}\|{\rm exp}\Big(\int_0^t\big(k+\frac{\|f(x(s))\|+\|f(x(s+h))\|}{r}\big)ds\Big) \; {\rm for\;all\;}t\ge 0.
\eeq
Fixed some $t_0\ge 0$ such that $\dot{x}(t_0)$ exists. Taking the limsup of both sides of $(\ref{estidiff})$ as $h\to 0^+$ and using Lemma \ref{limsup}, one gets 
$$\|\dot{x}(t_0)\|\le \|v_0\|{\rm exp}\Big(\int_0^{t_0}\big(k+\frac{2\|f(x(s))\|}{r}\big)ds\Big).$$
Thus $(\ref{boundedde})$ follows.
\end{proof}
Now, we are ready for the main result which  states that the solution is right differentiable and $\dot{x}^+(\cdot)$ is right continuous at each $t\ge 0$. We also recall an important property (Theorem \ref{mainth}-i) acquired in  Proposition 2.6 \cite{hm} by using a different approach. 
\begin{thm}\label{mainth}
Let $x(\cdot)$ be the unique solution of the system satisfying $x(0)=x_0$. Then we have:\\
$(i)$ $\dot{x}(t)=v(t)=\Big(f\big(x(t)\big)-N\big(C,x(t)\big)\Big)^0$ for almost every $t\in [0,+\infty).$\\
$(ii)$ For all $t^*\in [0,+\infty)$, the right derivative $\dot{x}^+(t^*)$ exists and 
$$\dot{x}^+(t^*)=\Big(f\big(x(t^*)\big)-N_C\big(x(t^*)\big)\Big)^0.$$
Furthermore $\dot{x}^+(\cdot)$ is continuous on the right.
\end{thm}
\begin{proof}
Let $E=\{t\in [0,+\infty): \dot{x}(t) \;{\rm exists}\}$. It is clear that the Lebesgue measure of $[0,+\infty)\setminus E$ is zero. \\
$(i)$ Fixed $t_0\in E$. Let $y(\cdot)$ be the unique solution of the system with initial condition $y(0)=x(t_0).$ Then $y(t)=x(t+t_0)$ for all $t\ge 0.$ Applying Lemma $\ref{limsup}$, we get
\beq
\limsup_{t\to 0^+}\|\frac{y(t)-y(0)}{t}\|\le \|\Big(f\big(y(0)\big)-N\big(C,y(0)\big)\Big)^0\|,
\eeq
or equivalently
\beq
\limsup_{t\to 0^+}\|\frac{x(t+t_0)-x(t_0)}{t}\|\le \|\Big(f\big(x(t_0)\big)-N\big(C,x(t_0)\big)\Big)^0\|.
\eeq
Hence 
\beq
\|\dot{x}(t_0)\|\le\|\Big(f\big(x(t_0))-N(C,x(t_0)\big)\Big)^0\|.
\eeq
On the other hand $\dot{x}(t_0)\in f\big(x(t_0)\big)-N\big(C,x(t_0)\big)$, thus $\dot{x}(t_0)=\Big(f\big(x(t_0)\big)-N\big(C,x(t_0)\big)\Big)^0.$\\
$(ii)$ Due to the property of semi-group, it is sufficient to prove for $t^*=0.$ Using $(i)$ and $(\ref{boundedde})$ of Lemma \ref{es2sol}, for all $t\in E$, we have
\beq
\|v(t)\|\le \|v_0\|{\rm exp}\Big(\int_0^t\big(k+\frac{2\|f(x(s))\|}{r}\big)ds\Big),
\eeq
where $v(t)= \Big(f\big(x(t)\big)-N\big(C,x(t)\big)\Big)^0$. It  implies that
\beq\label{a1}
\limsup_{t\to 0^+, t\in E}\|v(t)\|\le  \|v_0\|.
\eeq
On the other hand, Lemma \ref{lsc} deduces that 
\beq\label{a2}
 \|v_0\| \le \liminf_{t\to 0^+}\|v(t)\|\le \liminf_{t\to 0^+,t\in E}\|v(t)\|.
\eeq
From $(\ref{a1})$ and $(\ref{a2})$, we obtain
\beq
\lim_{t\to 0^+,t\in E}\|v(t)\|= \|v_0\|.
\eeq
Thus for any sequence $(t_n)_{n\ge 1}\subset E$ and $t_n\to 0$, we have
\beq\label{v_n}
\|v(t_n)\|\to \|v_0\| \;{\rm as}\; n\to +\infty.
\eeq
Then $\big(v(t_n)\big)_{n\ge 1}$ is bounded and therefore there exists some $v^*\in H$ such that a subsequence $(v(t_{n_k}))_{k\ge 1}$ converges weakly to $v^*$ when $k\to +\infty.$ Similarly as in Lemma \ref{lsc}, we can prove that $v^*\in f(x_0)-N(C;x_0)$. On the other hand
\beq
\|v^*\|\le \liminf_{k\to+\infty}\|v(t_{n_k})\|= \lim_{k\to+\infty}\|v(t_n)\|=\|v_0\|,
\eeq
due to $(\ref{v_n}).$ Thus, we must have $v^*=v_0$ and the set of weak cluster point of $\big(v(t_n)\big)_{n\ge 1}$ contains only $v_0$. It implies that $v(t_n)$ converges weakly to $v_0$. Combining with $(\ref{v_n})$, one deduces that $v(t_n)$ converges strongly to $v_0$. In conclusion
\beq\label{rcon}
\lim_{t\to 0^+,t\in E}v(t)=v_0.
\eeq
Due to the absolute continuity of $x(\cdot)$ and $(i)$, for all $h>0$, we have
\beq\label{differ}
x(h)-x_0=\int_0^h\dot{x}(s)ds=\int_0^hv(s)ds,
\eeq
where $v(\cdot)$ is locally integrable and satisfying $(\ref{rcon})$. Now we prove that 
\beq\label{aver}
\lim_{h\to 0^+}\frac{1}{h}\int_0^hv(s)ds=v_0.
\eeq
Indeed, given $\epsilon>0$. From $(\ref{rcon})$, there exists $\delta>0$ such that for all $s\in E, s\le \delta$ then $\|v(s)-v_0\|\le \epsilon$. Hence for all $h\le \delta$:
$$\|\frac{1}{h}\int_0^hv(s)ds-v_0\|\le \frac{1}{h}\int_0^h\|v(s)-v_0\|ds=\frac{1}{h}\int_{ [0,h]\cap E}\|v(s)-v_0\|ds\le\frac{\epsilon}{h}\int_{ [0,h]\cap E} ds=\epsilon.$$
So we have $(\ref{aver})$ and thus from $(\ref{differ})$, the right derivative $\dot{x}^+(0)$ exists and
\beq
\dot{x}^+(0)=v_0=\big(f(x_0)-N(C,x_0)\big)^0.
\eeq
It implies for all $t\ge 0$ that
\beq
\dot{x}^+(t)=v(t)=\Big(f\big(x(t)\big)-N\big(C,x(t)\big)\Big)^0.
\eeq
Then taking the limit both sides of $(\ref{estidiff})$, we deduce  for all $t\ge 0$ that
$$\|\dot{x}^+(t)\|\le \|\dot{x}^+(0)\|{\rm exp}\Big(\int_0^t\big(k+\frac{2\|f(x(s))\|}{r}\big)ds\Big),$$
or equivalently
$$\|v(t)\|\le \|v_0\|{\rm exp}\Big(\int_0^t\big(k+\frac{2\|f(x(s))\|}{r}\big)ds\Big).$$
Therefore
$$\limsup_{t\to 0^+}\|v(t)\|\le \|v_0\|.$$
Combining with $(\ref{a2})$, we obtain $\lim_{t\to 0^+}\|v(t)\|= \|v_0\|.$ Similar as $(\ref{rcon})$, we can prove that $\lim_{t\to 0^+}v(t)= v_0.$ It means that $\dot{x}^+(\cdot)$ is right continuous at $0$ and due to the property of semi-group, it is right continuous at any $t\ge 0.$
\end{proof}
Now we consider the case $f(\cdot)=-\nabla V(\cdot)$ where $V$ is $C^{1,+}$ function (i.e., $V$ is differentiable and $\nabla V$ is Lipschitz continuous) and study some asymptotic properties of the solutions. The system then can be considered as an extension of ``gradient equation" \cite{Aubin}.
\begin{prop}
Let  $V: H \to \R$ be a $C^{1,+}$ function. Let $x(\cdot)$ be the solution of the system
\begin{equation}
\left\{
\begin{array}{l}
\dot{x}(t) \in -\nabla V\big(x(t)\big) -N\big(C,x(t)\big)\; {\rm a.e.} \; t \in [0,+\infty),\\ \\
x(0) = x_0\in C.
\end{array}\right.
\end{equation}
Then we have\\
\beq\label{lya}
\frac{d}{dt}V\big(x(t)\big)+\|\dot{x}(t)\|^2=0, \;{\rm for}\; a. e. \;t\ge 0.
\eeq
In particular, $V$ is a Lyapunov function of the system. Furthermore\\
$(i)$ if $V$ is coercive, i.e.,
$$V(x)\to +\infty\;\;{\rm as}\;\; \|x\|\to +\infty,$$
 then $x(\cdot)$ is bounded on $\R_+$.\\
$(ii)$ if $V$ is bounded from below on $C$ then $\lim_{t\to+\infty}V\big(x(t)\big)=V_\infty$ exists and $\dot{x}\in  L^{2}([0,+\infty);H)$ with $\int_0^{+\infty}\|\dot{x}(s)\|^2ds=V(x_0)-V_\infty.$ \\
$(iii)$ If $V$ is convex and bounded from below on $C$ then $V_\infty=\displaystyle \inf_{y\in C} V(y).$
\end{prop}
\begin{proof}
 Fixed some $t\ge 0$ such that $(i)$ of Theorem \ref{mainth} holds, i.e., $\dot{x}(t)=\Big(-\nabla V\big(x(t)\big)-N_C\big(x(t)\big)\Big)^0=-\nabla V\big(x(t)\big)-{\rm proj}\Big(N\big(C,x(t)\big); -\nabla V\big(x(t)\big)\Big)$. Then
\baqn
&&\big\langle \dot{x}(t)+ \nabla V\big(x(t)\big),  \dot{x}(t) \big\rangle\\
&=&\big\langle -{\rm proj}\Big(N\big(C,x(t)\big); -\nabla V\big(x(t)\big)\Big), \nabla V(x(t))-{\rm proj}\Big(N\big(C,x(t)\big); -\nabla V\big(x(t)\big)\Big)\big\rangle\\
&=&0.
\eaqn
Note that $\frac{d}{dt}V\big(x(t)\big)=\big\langle\nabla V\big(x(t)\big),  \dot{x}(t) \big\rangle$ and $(\ref{lya})$ follows. In particular, we have $\frac{d}{dt}V\big(x(t)\big)\le 0$ for a.e. $t\ge 0.$ It means that $V$ is a Lyapunov function of the system.  Then $(i)$ and $(ii)$ follow classically.

%$(i)$ One has $V\big(x(t)\big)\le V(x_0),$ for all $t\ge 0$. Then the coerciveness of $V$ implies that $x(\cdot)$ is bounded on $\R_+.$
%
%$(ii)$ Since    $V\circ x(\cdot)$ is bounded from below and  non-increasing on $\R_+$, the limit $\lim_{t\to+\infty}V\big(x(t)\big)=V_\infty$ exists. Integrating both sides of $(\ref{lya})$ from $0$ to $t$, one has
%\beq
%\int_0^t\|\dot{x}(s)\|^2ds=V(x_0)-V\big(x(t)\big).
%\eeq
% Let $t\to +\infty$, we obtain $\int_0^{+\infty}\|\dot{x}(s)\|^2ds=V(x_0)-V_\infty.$\\
\noindent $(iii)$ Fix some $y\in C$ and consider the function  $\varphi(t)=\frac{1}{2}\|x(t)-y\|^2.$ Due to the $r$-prox-regularity of $C$ and the fact that $\dot{x}(t) +\nabla V\big(x(t)\big)\in -N\big(C,x(t)\big)\; {\rm a.e.} \; t \in [0,+\infty)$, one has 
$$\langle \dot{x}(t) +\nabla V\big(x(t)\big), x(t)-y \rangle\le \frac{\|\nabla V\big(x(t)\big)\|}{r}\|x(t)-y\|^2.$$ Thus
\baqn
\dot{\varphi}(t)=\langle \dot{x}(t), x(t)-y \rangle &\le& \frac{2\|\nabla V\big(x(t)\big)\|}{r} \varphi(t)+\langle\nabla V\big(x(t)\big), y-x(t) \rangle\\
&\le&  \frac{2\|\nabla V\big(x(t)\big)\|}{r} \varphi(t)+V(y)-V\big(x(t)\big),
\eaqn
due to the convexity of $V$. Using Gronwall's inequality (Lemma \ref{gronwall}), for all $t\ge 0$ one obtains 
\baqn
0\le\varphi(t)&\le& \varphi(0){\rm exp}\Big(\int_0^t \frac{2\|\nabla V(x(\tau))\|}{r}d\tau\Big)+\int_0^t {\rm exp}\Big(\int_s^t \frac{2\|\nabla V\big(x(\tau)\big)\|}{r}d\tau\Big)[V(y)-V\big(x(s)\big)]ds\\
&\le&{\rm exp}\Big(\int_0^t \frac{2\|\nabla V(x(\tau))\|}{r}d\tau\Big) \big[\varphi(0)+t\Big(V(y)-V\big(x(t)\big)\Big)\big],
\eaqn
since $ V\big(x(s)\big)\ge V\big(x(t))\big)$ for all $s\in [0,t]$.
It implies that 
$$V\big(x(t)\big)\le V(y)+\frac{\varphi(0)}{t}.$$
Let $t\to+\infty$, one gets $V_\infty\le V(y).$ Since $y$ is arbitrary in $C$, it deduces that $V_\infty\le\displaystyle \inf_{y\in C}V(y).$ On the other hand $V\big(x(t)\big)\ge\displaystyle\inf_{y\in C}V(y)$ since $x(t)\in C$ for all $t\ge 0.$ Hence $V_\infty\ge \displaystyle\inf_{y\in C}V(y)$. Therefore $V_\infty= \displaystyle\inf_{y\in C}V(y)$, it means the trajectory is minimizing for $V$ on $C$.\\

%$(iii)$ Denote $\mu$ the Lebesque measure on $\R$. Since $\dot{x}\in L^2([0,+\infty);H)$, it is easy to see that $\mu(\{t\ge 0: \|\dot{x}(t)\|<\varepsilon\})=+\infty$ for any given $\varepsilon>0.$ Indeed, if $\mu(\{t\ge 0: \|\dot{x}(t)\|<\varepsilon\})<+\infty$ then $\mu(X)=+\infty$ where $X=\{t\ge 0: \|\dot{x}(t)\|\ge\varepsilon\}$ and thus 
%$$\int_0^\infty \|\dot{x}(t)\|^2dt\ge \int_X \|\dot{x}(t)\|^2dt\ge \varepsilon^2\mu(X)=+\infty,$$
%a contradiction. Since $x_\infty$ is a strong limit point, there exists a sequence $(s_n)_{n\ge 1}\subset \R_+$ such that $s_n\to +\infty$ and $x(s_n)\to x_\infty.$ Applying Proposition 4.2 \cite{MT}, one has a strictly increasing mapping $\nu: \N\mapsto \N$ such that $\forall \delta >0, \exists m(\delta)\in \N, \forall n\ge m(\delta), \exists t_{n,\delta}\ge 0$ satisfying
%\begin{itemize}
%\item $s_{\nu(n)}-\delta< t_{n,\delta}< s_{\nu(n)}+\delta,$
%\item $\dot{x}(t_{n,\delta})$ exists,  $\dot{x}(t_{n,\delta})\in -\nabla V(x(t_{n,\delta}))-N_{C}(x(t_{n,\delta}))$ and 
%\item $\|\dot{x}(t_{n,\delta})\|\le \delta.$
%\end{itemize}
%Let $\delta\to 0, n\to +\infty$ then $\dot{x}(t_{n,\delta})\to 0$ and $x(t_{n,\delta})\to x_\infty$. Proving Similarly as in Lemma \ref{lsc}, one obtains $0\in -\nabla V(x_\infty) -N_{C}(x_\infty).$
\end{proof}

\section{Conclusion}\label{section4}
In this paper, we have established some important regularity properties for a class of differential inclusions involving normal cone operator of  prox-regular sets without tangential assumption. Some asymptotic behaviours of the solutions are also studied. It is interesting to consider properties of solutions of sweeping process with   prox-regular sets, where $C$ can depend on time and even the state. It is out of scope of the current work and will be considered in the future.\;
\vspace{5mm}

\noindent $\mathbf{Acknowledgments}$ \;\;The author would like to acknowledge the referees for their careful reading and insightful suggestions. The research is supported by Fondecyt Project 3150332.

\end{document}